\documentclass{article}
\usepackage{fouriernc}
\usepackage{fullpage}
\usepackage{hyperref}
\usepackage{url}
\usepackage{graphicx,wrapfig}
\usepackage{framed}

\usepackage{amsmath,amssymb,amsfonts,amsthm}
\usepackage{mkolar_definitions}

\title{Efficient Sparse Clustering of High-Dimensional Non-spherical Gaussian Mixtures}
\author{Martin Azizyan \qquad Aarti Singh \qquad Larry Wasserman\\ Carnegie Mellon University}

\begin{document}

\maketitle

\begin{abstract}
We consider the problem of clustering data points in high dimensions, i.e. 
when the number of data points may be much smaller than the number of 
dimensions.
Specifically, we consider a Gaussian mixture model (GMM) with non-spherical 
Gaussian components, where the clusters are distinguished by only a 
few relevant dimensions. The method we propose is a combination of a 
recent approach for learning parameters of a Gaussian mixture model 
and sparse linear discriminant analysis (LDA). In addition to cluster 
assignments, the method returns an estimate of the set of features 
relevant for clustering. Our results indicate that the sample complexity 
of clustering depends on the sparsity of the relevant feature set, while 
only scaling logarithmically with the ambient 
dimension. Additionally, we require much milder assumptions than 
existing work on clustering in high dimensions. In particular, we do not 
require spherical clusters nor necessitate mean separation along relevant dimensions. 
\end{abstract}

\section{Introduction}
The last few years have seen extensive research on developing efficient 
methods that can leverage sparsity of the relevant feature set for 
supervised learning (classification, regression etc.) of high-dimensional 
data. These methods show that learning and selection of relevant features 
is possible even if the number of training data $n$ is much less than the 
number of features $d$, which is typically the case for high-dimensional data. 
However, similar results for the unsupervised task of clustering are largely 
non-existent. The task of clustering high-dimensional data and extracting 
relevant features arises routinely in many applications, e.g., clustering of 
patients based on their gene expression profiles and identifying the relevant 
genotypes, grouping web content and identifying relevant characteristics, 
clustering proteins with similar drug expression profiles, etc. 

While there have been recent attempts at clustering high-dimensional data 
and selecting relevant 
features, these either do not come with theoretical 
guarantees or assume very strong conditions that suggest that even 
employing marginal feature selection, using projections of the data onto 
individual coordinates, as a pre-processing step before 
clustering might suffice. 
Thus, while supervised learning in high dimensions requires single-step 
methods that can perform the learning task and select relevant features 
simultaneously, it is not clear whether a sophisticated single-step approach 
is necessary for clustering in high dimensions. 

A simple example that will convince the reader that pre-processing the 
data using a marginal (coordinate-wise) feature selection step does not 
suffice for clustering, is provided by a mixture of two non-spherical 
Gaussian components (see Figure~\ref{fig:nonisoGMM}). It is clear that 
$x_1$ is relevant to define the clusters, however the marginal distribution 
of the data when projected onto $x_1$ is a single unimodal Gaussian. Hence, 
marginal feature selection cannot be used to identify the relevant features. 
\begin{figure}
\begin{center}
\includegraphics[scale = 0.5]{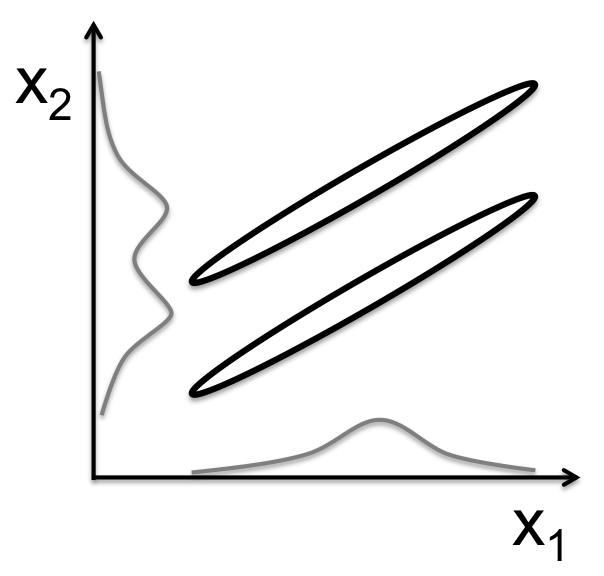}
\label{fig:nonisoGMM}
\caption{An example of two clusters where both features $x_1$ and $x_2$
are relevant to define the clusters, however the marginal distribution of 
data points along $x_1$ is unimodal, and hence no marginal feature selection 
method can work.}
\end{center}
\end{figure}
Motivated by this example, we consider a simple non-spherical Gaussian 
mixture model (defined formally in the next section) for clustering 
high-dimensional data, and aim to provide a computationally efficient 
algorithm for simultaneous feature 
selection and clustering, that comes with sample complexity guarantees 
that depend primarily on the number of relevant features (intrinsic dimension) 
and only logarithmically on the total number of features (ambient dimension). 

\noindent{\bf Related work.} Before we describe our approach and results, 
we discuss related work in some more detail. Sparse clustering methods that 
perform feature selection for high-dimensional data have received attention 
recently. 

K--means based approaches begin with the typical K--means objective 
and introduce some sparsity-inducing penalties \cite{sun2012regularized, witten2010framework, guo2010pairwise, andrews13selection,chang2014sparse}.
While the penalization introduced in these papers is convex (akin to supervised
learning approaches), the K--means objective itself is non-convex and in fact 
NP-hard. Thus, in general, solving any of these objectives is NP-hard and the 
papers propose iterative approaches akin to Llyod algorithm for solving the 
K--means objective. Moreover, these papers do not provide any statistical 
guarantees, with the exception of \cite{sun2012regularized,chang2014sparse}. 
The latter two papers do provide some consistency results, however these 
are for the true objective optimizers only which are NP-hard. 
Moreover, the notion of relevant features considered in all these papers
is that the means are separated along each relevant feature, which may not 
necessarily be the case as demonstrated in Figure~\ref{fig:nonisoGMM}. 

Another non-parametric approach to feature selection for clustering that is consistent in 
high dimensions is presented in \cite{hall2010clustering}, however it relies 
on pre-screening features which appear marginally unimodal, again failing 
for the example in Figure~\ref{fig:nonisoGMM}. 

Learning Gaussian mixture models (GMMs) has a long history, particularly in 
computer science theory community, where the emphasis has been on 
relaxing the assumptions under which GMMs can be learnt under various
metrics such as estimating the distribution, 
parameters or clustering. However, these papers primarily focus on computational 
tractability and mostly have high sample complexity, particularly in 
high dimensions. For example, the most relevant to this paper is the work 
on learning non-spherical GMMs where the components are separated by a 
hyperplane \cite{brubaker2008isotropic}, however it has sample complexity that depends 
as $d^4$ on the ambient dimension. The proposed estimator relies
on first making the data isotropic (zero mean and overall identity covariance). 
This is achieved by pre-whitening the data by multiplying it with the inverse
sample covariance matrix. However, in high dimensions when the number of samples
drawn from the mixture $n \ll d$ (the number of features), the sample covariance
matrix is not invertible and hence the method cannot succeed. 
Moreover, no work in this line, to the 
best of our knowledge, addresses feature selection. 
There is a very recent work \cite{hardt2014mixture} where the question of 
optimal sample complexity for GMM parameter estimation in $\ell_{\infty}$ 
norm is addressed and we build on this paper to provide 
clustering and feature selection guarantees. 

Apart from the work in the computer science theory community, multiple 
statistical approaches have also been proposed to learning Gaussian mixture 
models in high dimensions and feature selection \cite{maugis2009variable, 
lee2012variable, raftery2006variable, pan2007penalized, michel2014sparse, 
maugis2008non, maugis2008slope, Akshay_GMM, GMMinvCov, GMMinvCov2}.
These employ various sparsity assumptions, e.g.\ that the components are 
spherical and have sparse mean vectors, or that the covariance matrices (or their 
inverses) are also sparse, etc.
However, as with the K--means based methods, these approaches either a) require 
approximating maximum likelihood parameters without providing efficient 
algorithms; or b) do not come with precise finite sample statistical properties of 
the estimators. 

Assuming mixtures of equal weight spherical components with sparse mean separation, \cite{azizyan2013minimax} provide some minimax bounds for the problem with sample complexity that scales with the number of relevant features and only logarithmically 
with the total number of features. 
However, the assumption of spherical components necessitates that relevant features 
are characterized by mean separation, and hence a simple marginal variance thresholding procedure works for feature selection. Thus, the results do not apply for cases like the one
described in Figure~\ref{fig:nonisoGMM}.

Under less restrictive assumptions on the components, \cite{arias2014sparse-mixture} analyze detection of high-dimensional Gaussian mixtures (vs. a single Gaussian as null) and selection of sparse set of features along which mean separation occurs, from a minimax perspective. 
Their minimax optimal estimators involve combinatorial search, and while the authors also investigate 
some tractable procedures, they are either based on marginal feature selection or assume 
that the component covariance matrices are known and diagonal. 

Finally, we mention that if the cluster assignments are known, the problem of 
clustering reduces to binary classification. And
specifically, clustering using a mixture of two identical covariance Gaussians reduces
to linear discriminant analysis, if the cluster assignments are known. Feature selection 
using a sparse linear discriminant analysis has been analyzed in \cite{cai2011sparse-discriminant}. We will leverage this approach, in combination with the results for 
$\ell_\infty$ parameter estimation for GMMs \cite{hardt2014mixture} to demonstrate 
a method and results for sparse clustering and feature selection under high-dimensional 
non-spherical GMMs.

\noindent{\bf Contributions.} Our contributions can be summarized as follows:
\begin{itemize}
\item We present a computationally efficient method for clustering in high dimensions 
that comes with finite sample guarantees on the misclustering rate. Our results show that sample complexity scales quadratically 
with the number of relevant features (inherent dimension), but only logarithmically with total number of features (ambient dimension). As a result, the proposed method enables learning of non-spherical Gaussian mixtures of two components in high dimensions (when the 
number of data points may be much smaller than the number of features), without assuming 
sparsity of the covariance or inverse covariance matrix. 
\item We provide guarantees for feature selection under a very generalized notion of relevant features that does not require that clusters 
necessarily have mean separation along the relevant features. This allows us to handle cases like that shown 
in Figure~\ref{fig:nonisoGMM}. 
\end{itemize}

The rest of the paper is organized as follows. In section~\ref{sec:setup}, we formalize our setup. The proposed method combining ideas from \cite{hardt2014mixture} and \cite{cai2011sparse-discriminant} is presented in section~\ref{sec:method}. Section~\ref{sec:results} states our results on misclustering rate, sample complexity and feature selection in high dimensions. We conclude with some open directions in section~\ref{sec:concl}.

\section{Problem Setup and Assumptions}
\label{sec:setup}
Inspired by Figure~\ref{fig:nonisoGMM}, we consider the following simple model. 
\begin{itemize}
\item [A1)] {\em Data generating model:} The data points $X_1, \dots, X_n$ are generated i.i.d. from a mixture of two Gaussians
of the form $\frac{1}{2}\Ncal(\mu_1,\Sigma)+\frac{1}{2}\Ncal(\mu_2,\Sigma)$ in $\RR^d$. 
\end{itemize}
Extensions to unequal weights and covariances, and more than 2 clusters are interesting, 
but out of the scope of this paper. 

The error of a clustering $\psi:\RR^d\rightarrow\{1,2\}$ is defined as follows.
Let $X$ be a random draw from the true mixture, and let $Y\in\{1,2\}$ be the (latent) label of the mixture component from which $X$ was drawn, i.e. $Y-1\sim\mathrm{Bern}\rbr{\frac{1}{2}}$ and $X|Y\sim\Ncal(\mu_Y,\Sigma)$.
We define the {\em overlap} of the clustering $\psi$ as $\Upsilon(\psi):=\min_\pi\PP(\psi(X)\neq \pi(Y))$ where the minimum is over permutations $\pi:\{1,2\}\rightarrow\{1,2\}$, and the error of $\psi$ is defined as $L(\psi):=\Upsilon(\psi)-\min_{\psi'}\Upsilon(\psi')$.

We define the optimal clustering $\psi^\ast:=\argmin_{\psi}\Upsilon(\psi)$, which coincides with the Bayes optimal classifier in the supervised problem of prediting $Y$ from $X$:
\begin{equation}
\psi^\ast(x)=
\left\{ \begin{array}{rcl}
1 & \mbox{if} & (\mu_0-x)^\tp\Sigma^{-1}\Delta_\mu<0, \\
2 & \mbox{o.w.} & 
\end{array}
\right.
\label{eq:opt}
\end{equation}
where $\mu_0=\frac{\mu_1+\mu_2}{2}$ and $\Delta_\mu=\frac{\mu_1-\mu_2}{2}$.
Notice that the Bayes optimal decision boundary is linear and hence the problem 
corresponds to linear discriminant analysis (LDA). 

If the labels are known, one can simply plug-in sample estimates of class 
conditional means $\widehat \mu_Y$ and covariance matrix $\widehat \Sigma$ 
to obtain an empirical classification rule. In clustering, the labels are latent. 
However, if we can learn the parameters $\mu_1,\mu_2,\Sigma$ of the Gaussian 
mixture model, we can plug these in and obtain a similar empirical clustering. 

In the high-dimensional setting $(n \ll d)$, estimates of the covariance matrix 
are typically not invertible, necessitating some additional assumptions to make 
the problem well-posed. In high-dimensional clustering, it is natural to expect 
that not all features are relevant for clustering. For example, in clustering proteins 
based on their drug expression profiles, not all drugs are responsible for differentiation
of the expressions. This assumption can be captured as follows. 
\begin{itemize}
\item [A2)] {\em Sparsity of relevant features: }A feature is 
relevant if the optimal clustering rule $\psi^*$ in Eq.~\ref{eq:opt} depends on it. Equivalently, relevant features 
$S\subseteq \{1,\dots, d\}$ 
are given by non-zero coordinates of $\Sigma^{-1}\Delta_\mu$, where 
$|S| \leq s \leq d$. 
\end{itemize}
We will 
demonstrate that the sample complexity of clustering in high dimensions depends 
on the number of non-zero coordinates $\|\Sigma^{-1}\Delta_\mu\|_0 = |S| \leq s$, 
and only logarithmically on the total number of features $d$. 

In comparison, existing work on high-dimensional clustering typically assumes that $\Delta_\mu$ is sparse and the relevant features are given by non-zero coordinates 
of $\Delta_\mu$, i.e., the coordinates along which mean separation occurs. Thus, 
they cannot identify relevant features such as $x_1$ in Figure~\ref{fig:nonisoGMM}. 
Also, some existing work on high-dimensional GMM learning assumes sparsity of 
the covariance $\Sigma$ or inverse covariance matrix $\Sigma^{-1}$. 
These assumptions used in previous work are more restrictive than (and can be considered
as special cases of) 
our notion of 
relevant features (non-zero coordinates of $\Sigma^{-1}\Delta_\mu$) which is 
precisely what the optimal clustering function depends on. 

We make the following additional assumption which guarantees 
success of our computationally feasible method that uses the $\ell_1$ penalty.
\begin{itemize}
\item[A3)]  {\em Restricted eigenvalue property:} The covariance matrix $\Sigma$ 
satisfies 
$$\min_{S\subseteq \{1,\dots,d\} : |S|\leq s}\min_{v\neq 0} \left\{\frac{\|\Sigma v\|_2}{\|v\|_2}:
\|v_{S^c}\|_1 \leq \|v_S\|_1\right\} \geq \eta >0$$
\end{itemize}
Similar assumption is required for feature selection in supervised learning using 
$\ell_1$ penalties (c.f. \cite{Bickel09}).

While the above assumptions suffice to evaluate the clustering performance
in high dimensions, we also seek to correctly identify the set of relevant features. 
For this, we need to assume that each relevant feature is ``relevant enough" to be 
detectable using a finite sample. Formally,
\begin{itemize}
\item [A4)] {\em Signal strength along each relevant feature:} Let $\beta = \Sigma^{-1}\Delta_\mu$. Then $\beta_{\min} := \min_i \beta(i)
\geq$ $C' s\left(\frac{\log d}{n}\right)^{1/6}$, where $C'>0$ is a constant.
\end{itemize}

\section{Proposed Method}
\label{sec:method}

Given samples $X_1, \dots, X_n$ from the unknown mixture $\frac{1}{2}\Ncal(\mu_1,\Sigma)+\frac{1}{2}\Ncal(\mu_2,\Sigma)$, we propose a procedure composed of three stages.
First we aquire initial estimates of mixture parameters using the algorithm of Hardt and Price \cite{hardt2014mixture}.
Next we estimate the discriminating direction $\beta := \Sigma^{-1}\Delta_\mu$ by means of solving a convex program analogous to the proposal of Cai and Liu \cite{cai2011sparse-discriminant} for sparse supervised linear classification.
Finally we threshold the elements of the estimate of $\beta$ to recover the relevant features $S$.

Precisely, the steps are as follows.
\begin{enumerate}
\item Obtain estimates $\widehat{\mu}_1,\widehat{\mu}_2,\widehat{\Sigma}$ by invoking Algorithm \textsc{GMFitHardtPrice} defined in Section~\ref{sec:algo-hardt} 
 with $\epsilon,\delta$  satisfying $\epsilon= C(\log(dn/\delta)/n)^{1/6}$ for some constant $C$.
\end{enumerate}

Algorithm \textsc{GMFitHardtPrice} assumes the availability of a separate algorithm for mixture learning in a low-dimensional setting.
Specifically, \textsc{GMFitHardtPrice} uses an algorithm \textsc{GMFitLowDim} that also takes as input a set of samples as described in (A1) and parameters $\epsilon,\delta>0$, and has the same properties as those of \textsc{GMFitHardtPrice} stated in Theorem~\ref{thm:gmm}, but only for up to $2$ dimensional mixtures\footnote{Since we are assuming the components of the mixture share the same covariance, the algorithms appear slightly simplified compared to those described by Hardt and Price}.
Hardt and Price give one candidate for \textsc{GMFitLowDim}, which combines a type of moment method approach and a grid search over parameters.
The algorithm involves a large number of steps, and we do not restate it here due to the space constraint.
Since much of the computational and statistical difficulty of general Gaussian mixture learning is not present when only considering such small dimensional cases, we believe this does not take away from the exposition.


\begin{enumerate}
\setcounter{enumi}{1}
\item For some $\lambda>0$, set
\begin{align}
\widehat{\beta}_\lambda=\argmin_{z\in\RR^d}\|z\|_1 \mathrm{\;subject\,to\;} \|\widehat{\Sigma}z-\widehat{\Delta}_\mu\|_\infty\leq\lambda \label{eq:def-opti}
\end{align}
 where $\widehat{\mu}_0=\frac{\widehat{\mu}_1+\widehat{\mu}_2}{2}$, $\widehat{\Delta}_\mu=\frac{\widehat{\mu}_1-\widehat{\mu}_2}{2}$, $\|\cdot\|_\infty$ is the elementwise absolute maximum, and $\lambda$ is a tuning parameter the choice of which is discussed below.
The estimated clustering is defined as
$$
\widehat{\psi}_\lambda(x)=
\left\{ \begin{array}{rcl}
1 & \mbox{if} & (\widehat{\mu}_0-x)^\tp\widehat{\beta}_\lambda<0, \\
2 & \mbox{o.w.} & 
\end{array}
\right.
$$
\end{enumerate}
Proposition~\ref{prop:main} in Section~\ref{sec:results} ties the error in the estimates of $\widehat{\Sigma}$ and $\widehat{\Delta}_\mu$ to the error of $\widehat{\psi}_\lambda$.
In this result, the bound on the clustering error is minimized when $\lambda$ takes on the smallest value such that the true $\beta$ is a feasible point of the constraints in (\ref{eq:def-opti}).
A specific value for $\lambda$ is given in Corollary~\ref{cor:nclerr}, based on a few additional technical assumptions.
\begin{enumerate}
\setcounter{enumi}{2}
\item 
Estimate the relevant features $S$ by thresholding $\widehat \beta_\lambda$:
$$
\widehat S = \{i: \widehat \beta_\lambda(i) > c\cdot\lambda \sqrt{s}\}
$$
where $c > 0$ is a constant.
\end{enumerate}
Our results for support recovery hold when $c>2/\eta$.

\subsection{Algorithm \textsc{GMFitHardtPrice}.}
\label{sec:algo-hardt}
\textbf{Input:} Samples $X_1,\ldots,X_n\in\RR^d$; $\epsilon,\delta>0$.

\begin{enumerate}
\item Let $\widehat{V}=\max_{i\in[d]} \frac{1}{n}\sum_{j=1}^n X_j(i)^2-\rbr{\frac{1}{n}\sum_{j=1}^n X_j(i)}^2$.
Set $\epsilon^\ast=\epsilon/20$ and $\delta^\ast=\delta/(10d^2)$.
Algorithm \textsc{GMFitLowDim} will always be invoked with parameters $\epsilon^\ast$ and $\delta^\ast$.
\end{enumerate}

\textbf{Estimate $\widehat{\mu}_1$ and $\widehat{\mu}_2$:}
\begin{enumerate}
\setcounter{enumi}{1}
\item For each $i\in[d]$ (where $[d]=\{1,\ldots,d\}$), use \textsc{GMFitLowDim} on the univariate data $X_1(i),\ldots,X_n(i)$ obtaining estimates of the means $\xi_1(i)$ and $\xi_2(i)$ (discard the variance estimates).
\item If $|\xi_1(i)-\xi_2(i)|\leq\epsilon\widehat{V}/4$ for all $i\in[d]$, put $\widehat{\mu}_1=\widehat{\mu}_2=\xi_1$ (and skip step 4).
\item Otherwise, let $i$ be the smallest index such that $|\xi_1(i)-\xi_2(i)|>\epsilon\widehat{V}/4$ and, for each $j\in[d]\setminus\{i\}$ do:
\begin{itemize}
\item[a)] Apply \textsc{GMFitLowDim} to the bivariate data $[X_1(i),X_1(j)],\ldots,[X_n(i),X_n(j)]$ to obtain mean estimates $(\nu_k(i),\nu_k(j))$ for $k=1,2$.
\item[b)] Let $k\in\{1,2\}$ such that $|\xi_1(i)-\nu_k(i)|\leq\epsilon\widehat{V}/10$.
If such $k$ does not exist, the algorithm terminates with failure.
\item[c)] Set $\widehat{\mu}_1(j)=\nu_k(j)$ and $\widehat{\mu}_2(j)=\nu_{3-k}(j)$.
\end{itemize}
\end{enumerate}
\textbf{Estimate $\widehat{\Sigma}$:}
\begin{enumerate}
\setcounter{enumi}{4}
\item For each $i\in[d]$, invoke \textsc{GMFitLowDim} on the univariate data $X_1(i),\ldots,X_n(i)$ and obtain an estimate of the diagonal element $\widehat{\Sigma}(i;i)$ (discarding the estimates of the means).
\item For each $i<j\in[d]$, invoke \textsc{GMFitLowDim} on the bivariate data $[X_1(i),X_1(j)],\ldots,[X_n(i),X_n(j)]$ and obtain an estimate of $\widehat{\Sigma}(i,j)=\widehat{\Sigma}(j,i)$.
\item Return $\widehat{\mu}_1,\widehat{\mu}_2,\widehat{\Sigma}$.
\end{enumerate}

\section{Main Result}
\label{sec:results}

Our first result states that if the parameters of the Gaussian mixture model in (A1) 
can be learnt accurately in $\ell_\infty$ norm, then the misclustering rate of the 
proposed method is small. 

\begin{proposition}
Assume (A1).
For any $\epsilon$, if $\max\rbr{\|\mu_1-\widehat{\mu}_{\pi(1)}\|_\infty^2,\|\mu_2-\widehat{\mu}_{\pi(2)}\|_\infty^2,\|\Sigma-\widehat{\Sigma}\|_\infty}\leq \epsilon$ for some permutation $\pi:\{1,2\}\rightarrow\{1,2\}$, and if $\epsilon\|\beta\|_1+\sqrt{\epsilon}\leq\lambda$, then
\begin{align*}
L(\widehat{\psi}_\lambda)
&\leq
\phi\rbr{\max\rbr{\frac{\Delta_\mu^\tp\Sigma^{-1}\Delta_\mu-\epsilon_1}{\sqrt{\Delta_\mu^\tp\Sigma^{-1}\Delta_\mu+\epsilon_2}},\; 0}}  \frac{
\epsilon_1+\epsilon_2 }{\sqrt{\Delta_\mu^\tp\Sigma^{-1}\Delta_\mu}}
\end{align*}
where $\epsilon_1=(2\lambda+3\sqrt{\epsilon})\|\beta\|_1$, $\epsilon_2=\epsilon\|\beta\|_1^2 + 3(\lambda+\sqrt{\epsilon})\|\beta\|_1$, and $\phi$ is the standard normal density.
\label{prop:main}
\end{proposition}
Before giving the proof, we notice that the misclustering rate depends on 
$\Delta_\mu^\tp\Sigma^{-1}\Delta_\mu$ which can be regarded as the signal energy. 
\begin{proof}
Since the clustering error does not change upon flipping the labels assigned by $\widehat{\psi}_\lambda$, WLOG we assume $\pi$ is the identity permutation.

It is easily to verify that
$$
\Upsilon(\psi^\ast)=\Phi\rbr{-\frac{\Delta_\mu^\tp\beta}{\sqrt{\beta^\tp\Sigma\beta}}}=\Phi\rbr{-\sqrt{\Delta_\mu^\tp\Sigma^{-1}\Delta_\mu}}
$$
where $\Phi$ is the standard normal CDF.
Also,
$$
\Upsilon(\widehat{\psi}_\lambda)=\frac{1}{2}\Phi\rbr{-\frac{|\Delta_\mu^\tp\widehat{\beta}_\lambda|+|(\mu_0-\widehat{\mu}_0)^\tp\widehat{\beta}_\lambda|}{\sqrt{\widehat{\beta}_\lambda^\tp\Sigma\widehat{\beta}_\lambda}}}+\frac{1}{2}\Phi\rbr{-\frac{|\Delta_\mu^\tp\widehat{\beta}_\lambda|-|(\mu_0-\widehat{\mu}_0)^\tp\widehat{\beta}_\lambda|}{\sqrt{\widehat{\beta}_\lambda^\tp\Sigma\widehat{\beta}_\lambda}}}
$$
where the appearance of the absolute values is to account for the minimum over permutations in the definition of $\Upsilon$ -- the overlap must not change if $\widehat{\beta}_\lambda$ is negated.
So,
\begin{align}
L(\widehat{\psi}_\lambda)=\Upsilon(\widehat{\psi}_\lambda)-\Upsilon(\psi^\ast)
\leq
\Phi\rbr{-\frac{|\Delta_\mu^\tp\widehat{\beta}_\lambda|-|(\mu_0-\widehat{\mu}_0)^\tp\widehat{\beta}_\lambda|}{\sqrt{\widehat{\beta}_\lambda^\tp\Sigma\widehat{\beta}_\lambda}}} 
-\Phi\rbr{-\sqrt{\Delta_\mu^\tp\Sigma^{-1}\Delta_\mu}}. \label{eq:loss-gen-bound}
\end{align}

Clearly, $\|\mu_0-\widehat{\mu}_0\|_\infty\leq\sqrt{\epsilon}$ and $\|\Delta_\mu-\widehat{\Delta}_\mu\|_\infty\leq\sqrt{\epsilon}$.
Since $\Delta_\mu=\Sigma\beta$,
\begin{align*}
\|\widehat{\Sigma}\beta-\widehat{\Delta}_\mu\|_\infty
&\leq \|\widehat{\Sigma}\beta-\Delta_\mu\|_\infty + \|\Delta_\mu-\widehat{\Delta}_\mu\|_\infty \\
&\leq \|\widehat{\Sigma}\beta-\Sigma\beta\|_\infty + \sqrt{\epsilon}\\
&\leq \|\widehat{\Sigma}-\Sigma\|_\infty\|\beta\|_1 + \sqrt{\epsilon} \\
&\leq \epsilon\|\beta\|_1+\sqrt{\epsilon}\leq\lambda
\end{align*}
which implies that $\beta$ is a feasible point for the optimization problem (\ref{eq:def-opti}).
Hence, since $\widehat{\beta}_\lambda$ is an optimum for (\ref{eq:def-opti}), $\|\widehat{\beta}_\lambda\|_1\leq\|\beta\|_1$, and
$$
|(\mu_0-\widehat{\mu}_0)^\tp\widehat{\beta}_\lambda| \leq 
\|\mu_0-\widehat{\mu}_0\|_\infty \|\widehat{\beta}_\lambda\|_1
\leq \sqrt{\epsilon} \|\beta\|_1.
$$
Next,
\begin{align*}
|\Delta_\mu^\tp\widehat{\beta}_\lambda|&\geq
|\Delta_\mu^\tp\beta|-|\Delta_\mu^\tp(\widehat{\beta}_\lambda-\beta)| \\
&=\Delta_\mu^\tp\Sigma^{-1}\Delta_\mu-|\Delta_\mu^\tp(\widehat{\beta}_\lambda-\beta)|
\end{align*}
where
\begin{align*}
|\Delta_\mu^\tp(\widehat{\beta}_\lambda-\beta)|
&\leq |\widehat{\beta}_\lambda^\tp(\widehat{\Sigma}\beta-\Delta_\mu)|
+ |\beta^\tp(\widehat{\Sigma}\widehat{\beta}_\lambda-\Delta_\mu)| \\
&\leq \|\beta\|_1 \rbr{ \|\widehat{\Sigma}\beta-\Delta_\mu\|_\infty
+ \|\widehat{\Sigma}\widehat{\beta}_\lambda-\Delta_\mu\|_\infty } \\
&\leq \|\beta\|_1 \rbr{ \|\widehat{\Sigma}\beta-\widehat{\Delta}_\mu\|_\infty
+ \|\widehat{\Sigma}\widehat{\beta}_\lambda-\widehat{\Delta}_\mu\|_\infty +2\|\Delta_\mu-\widehat{\Delta}_\mu\|_\infty} \\
&\leq 2(\lambda+\sqrt{\epsilon})\|\beta\|_1
\end{align*}
i.e.
$$
\Delta_\mu^\tp\Sigma^{-1}\Delta_\mu - \rbr{
|\Delta_\mu^\tp\widehat{\beta}_\lambda|
- |(\mu_0-\widehat{\mu}_0)^\tp\widehat{\beta}_\lambda| }
\leq
(2\lambda+3\sqrt{\epsilon})\|\beta\|_1 \equiv\epsilon_1.
$$
And
\begin{align*}
\widehat{\beta}_\lambda^\tp\Sigma\widehat{\beta}_\lambda &\leq 
\beta^\tp\Sigma\beta + |\widehat{\beta}_\lambda^\tp\Sigma\widehat{\beta}_\lambda-\beta^\tp\Sigma\beta| \\
&\leq
\beta^\tp\Sigma\beta + |\widehat{\beta}_\lambda^\tp\Sigma\widehat{\beta}_\lambda-\widehat{\beta}_\lambda^\tp\Delta_\mu|+|(\widehat{\beta}_\lambda-\beta)^\tp\Delta_\mu| \\
&\leq
\Delta_\mu^\tp\Sigma^{-1}\Delta_\mu + \|\Sigma\widehat{\beta}_\lambda-\Delta_\mu\|_\infty\|\beta\|_1+2(\lambda+\sqrt{\epsilon})\|\beta\|_1
\end{align*}
where
\begin{align}
\|\Sigma\widehat{\beta}_\lambda-\Delta_\mu\|_\infty
&\leq 
\|\Sigma\widehat{\beta}_\lambda-\widehat{\Sigma}\widehat{\beta}_\lambda\|_\infty+\|\widehat{\Sigma}\widehat{\beta}_\lambda-\widehat{\Delta}_\mu\|_\infty+\|\widehat{\Delta}_\mu-\Delta_\mu\|_\infty \nonumber\\
&\leq \|\Sigma-\widehat{\Sigma}\|_\infty\|\beta\|_1+\lambda+\sqrt{\epsilon} \nonumber\\
&\leq \epsilon\|\beta\|_1+\lambda+\sqrt{\epsilon}
\label{eqn}
\end{align}
so
$$
\widehat{\beta}_\lambda^\tp\Sigma\widehat{\beta}_\lambda
- 
\Delta_\mu^\tp\Sigma^{-1}\Delta_\mu
\leq
 \epsilon\|\beta\|_1^2 + 3(\lambda+\sqrt{\epsilon})\|\beta\|_1 \equiv\epsilon_2.
$$

Combining these with (\ref{eq:loss-gen-bound}),
\begin{align*}
L(\widehat{\psi}_\lambda)&\leq
\Phi\rbr{-\frac{\Delta_\mu^\tp\Sigma^{-1}\Delta_\mu-\epsilon_1}{\sqrt{\Delta_\mu^\tp\Sigma^{-1}\Delta_\mu+\epsilon_2}}} 
-\Phi\rbr{-\sqrt{\Delta_\mu^\tp\Sigma^{-1}\Delta_\mu}} \\
&\leq
\phi\rbr{\max\rbr{\frac{\Delta_\mu^\tp\Sigma^{-1}\Delta_\mu-\epsilon_1}{\sqrt{\Delta_\mu^\tp\Sigma^{-1}\Delta_\mu+\epsilon_2}},\; 0}}
\rbr{
\sqrt{\Delta_\mu^\tp\Sigma^{-1}\Delta_\mu}
- \frac{\Delta_\mu^\tp\Sigma^{-1}\Delta_\mu-\epsilon_1}{\sqrt{\Delta_\mu^\tp\Sigma^{-1}\Delta_\mu+\epsilon_2}}} \\
&\leq
\phi\rbr{\max\rbr{\frac{\Delta_\mu^\tp\Sigma^{-1}\Delta_\mu-\epsilon_1}{\sqrt{\Delta_\mu^\tp\Sigma^{-1}\Delta_\mu+\epsilon_2}},\; 0}}  \frac{
\epsilon_1+\epsilon_2 }{\sqrt{\Delta_\mu^\tp\Sigma^{-1}\Delta_\mu}}.
\end{align*}
\end{proof}

The following result from \cite{hardt2014mixture} provides us $\ell_\infty$
control over the GMM parameters. 

\begin{theorem}[Hardt and Price \cite{hardt2014mixture}]
Given $\epsilon,\delta>0$ and $n$ samples from the model (A1), if
$$n=O\rbr{\frac{1}{\epsilon^{6}}\log\rbr{\frac{d}{\delta}\log\rbr{\frac{1}{\epsilon}}}},$$
then, with probability at least $1-\delta$, Algorithm \textsc{GMFitHardtPrice} defined in Section~\ref{sec:algo-hardt} 
 produces parameter estimates $\widehat{\mu}_1,\widehat{\mu}_2$ and $\widehat{\Sigma}$ such that, for some permutation $\pi:\{1,2\}\rightarrow\{1,2\}$,
$$
\max\rbr{\max_{i=1,2}\rbr{\|\mu_i-\widehat{\mu}_{\pi(i)}\|_\infty^2},\;\|\Sigma-\widehat{\Sigma}\|_\infty}\leq\epsilon\rbr{ \frac{1}{4}\|\mu_1-\mu_2\|_\infty^2+ \|\Sigma_1\|_\infty}.
$$
\label{thm:gmm}
\end{theorem}

Combining Proposition~\ref{prop:main} and Theorem~\ref{thm:gmm}, we have the 
following result under (A2) and (A4).
We defer the proof to the supplement.

\begin{corollary}\label{cor:nclerr}
Assume (A1), (A2), (A3), $\|\mu_1-\mu_2\|_\infty^2<D$ and $\|\Sigma\|_2\leq D_0$.
Given $\delta>0$, there is some constant $c_1$ such that, setting 
$$
\lambda=
c_1\left(\frac{\log(dn/\delta)}{n}\right)^{1/6}
\frac{\sqrt{D_0s (\Delta_\mu^\tp\Sigma^{-1}\Delta_\mu)}}{\eta}
+\sqrt{c_1}\left(\frac{\log(dn/\delta)}{n}\right)^{1/12},
$$
with probability at least $1-\delta$,
\begin{align*}
L(\widehat{\psi}_\lambda) 
\leq&
C_0
\phi \rbr{\sqrt{\frac{\Delta_\mu^\tp\Sigma^{-1}\Delta_\mu}{6}}}
\max\rbr{
\frac{s\sqrt{\Delta_\mu^\tp\Sigma^{-1}\Delta_\mu}}{\eta^2}
\left(\frac{\log(dn/\delta)}{n}\right)^{1/6}
,\,
 \frac{\sqrt{s}}{\eta}
\left(\frac{\log(dn/\delta)}{n}\right)^{1/12} }
\end{align*}
for some constant $C_0$.
\end{corollary}
\begin{proof}
Applying Theorem~\ref{thm:gmm} with $\epsilon = (\log(dn/\delta)/n)^{1/6}$, since $\|\Sigma\|_\infty\leq\|\Sigma\|_2\leq D_0$, we have, for some constant $c_1$ which depends on $D$ and $D_0$,  with probability at least $1-\delta$
$$
\max\rbr{\|\mu_1-\widehat{\mu}_{\pi(1)}\|_\infty^2,\|\mu_2-\widehat{\mu}_{\pi(2)}\|_\infty^2,\|\Sigma-\widehat{\Sigma}\|_\infty} = c_1\left(\frac{\log(dn/\delta)}{n}\right)^{1/6}=:\epsilon_0.
$$

For concision, let $\rho=\Delta_\mu^\tp\Sigma^{-1}\Delta_\mu$.
From Proposition~\ref{prop:main}, if $\epsilon_0\|\beta\|_1+\sqrt{\epsilon_0}\leq\lambda$, then
\begin{align*}
L(\widehat{\psi}_\lambda)
&\leq
\phi\rbr{\max\rbr{\frac{\rho-\epsilon_1}{\sqrt{\rho+\epsilon_2}},\; 0}}  \frac{
\epsilon_1+\epsilon_2 }{\sqrt{\rho}}
\end{align*}
where $\epsilon_1=(2\lambda+3\sqrt{\epsilon_0})\|\beta\|_1$ and $\epsilon_2=\epsilon_0\|\beta\|_1^2 + 3(\lambda+\sqrt{\epsilon_0})\|\beta\|_1$.
Now, using (A2), (A3), and $\|\Sigma\|_2\leq D_0$, we have
$$
\|\beta\|_1
\leq \sqrt{s}\|\beta\|_2
\leq\frac{\sqrt{s}}{\eta}\|\Sigma\beta\|_2
\leq\frac{\sqrt{D_0s}}{\eta}\sqrt{\rho}.
$$
Taking $\lambda=\epsilon_0\frac{\sqrt{D_0s\rho}}{\eta}+\sqrt{\epsilon_0}$, we have that the condition in Proposition~\ref{prop:main} is satisfied.
Also, defining $\omega=\frac{\epsilon_0D_0s}{\eta^2}$,
$$
\epsilon_1
\leq (2\lambda+3\sqrt{\epsilon_0}) \frac{\sqrt{D_0s}}{\eta}\sqrt{\rho}
= 2\omega\rho+5\sqrt{\omega\rho}
$$
and
$$
\epsilon_2
\leq \epsilon_0\rbr{\frac{\sqrt{D_0s}}{\eta}\sqrt{\rho}}^2 + 3(\lambda+\sqrt{\epsilon_0})\frac{\sqrt{D_0s}}{\eta}\sqrt{\rho}
= 4\omega\rho + 6\sqrt{\omega\rho}
$$
Thus
\begin{align*}
L(\widehat{\psi}_\lambda) 
&\leq
\phi\rbr{\max\rbr{\frac{\rho-\epsilon_1}{\sqrt{\rho+\epsilon_2}},\; 0}}
\rbr{
6\sqrt{\omega\rho}+11
} \sqrt{\omega}.
\end{align*}
The ratio in $\phi$ is lower bounded by $\sqrt{\rho/6}$ as long as $n$ is sufficiently large so that $\epsilon_1\leq\epsilon_2\leq\rho/2$, but since $L(\widehat{\psi}_\lambda)\leq 1$ always, the result follows for all $n$ assuming $C_0$ sufficiently large.
\end{proof}

\paragraph{Remark:}
Hence it follows that $n=\Omega\rbr{s^6\log(d)}$, suppressing the dependence on other parameters.

\subsection{Recovery of relevant features}

We derive a bound for $\|\beta-\widehat{\beta}_\lambda\|_\infty$ under (A3) and then 
guarantee recovery of relevant features under (A4), i.e. when the non-zero components of $\beta$ are large enough. 
\begin{theorem} Assume the conditions of Proposition~\ref{prop:main} hold. 
We have
$$
\|\beta- \widehat{\beta}_\lambda\|_\infty \leq \frac{2\lambda\sqrt{s}}{\eta}
$$
If, in addition, (A4) holds and $\eta > 2/c$ then 
$$
\widehat S = S.
$$
\end{theorem}
\begin{proof}
We first establish two results that are crucial. First, 
if $\beta$ is sparse and $S$ denotes the support of $\beta$, then
\[
\|(\beta - \widehat{\beta}_\lambda)_{S^c}\|_1 = \|\widehat{\beta}_{\lambda,{S^c}}\|_1
= \|\widehat{\beta}_{\lambda}\|_1 - \|\widehat{\beta}_{\lambda,{S}}\|_1
\leq \|\beta\|_1 - \|\widehat{\beta}_{\lambda,{S}}\|_1
 = \|\beta_S\|_1 - \|\widehat{\beta}_{\lambda,{S}}\|_1
\leq \|(\beta - \widehat{\beta}_\lambda)_{S}\|_1
\]

Second, note that
\[
\|\Sigma(\beta-\widehat{\beta}_\lambda)\|_\infty 
= \|\Delta_\mu - \Sigma\widehat{\beta}_\lambda\|_\infty \leq 2 \lambda
\]
using Eq.~(\ref{eqn}).

Therefore, we can write
\begin{eqnarray*}
\|\beta- \widehat{\beta}_\lambda\|_\infty \leq \|\beta- \widehat{\beta}_\lambda\|_2 
& \leq & \frac{2\lambda}{\|\Sigma(\beta-\widehat{\beta}_\lambda)\|_\infty}\|\beta- \widehat{\beta}_\lambda\|_2
\leq \frac{2\lambda\sqrt{s}}{\|\Sigma(\beta-\widehat{\beta}_\lambda)\|_2}\|\beta- \widehat{\beta}_\lambda\|_2\\
& \leq & \frac{2\lambda\sqrt{s}}{\min_v \{\|\Sigma v\|_2/\|v\|_2:
\|v_{S^c}\|_1 \leq \|v_S\|_1\}} \leq \frac{2\lambda\sqrt{s}}{\eta}
\end{eqnarray*}

The result follows using (A4). 
\end{proof}

\section{Discussion and Open questions}
\label{sec:concl}
The primary goal of this paper was to demonstrate a method for
high-dimensional clustering which, in contrast to existing work,
provably identifies relevant features that cannot be detected by
marginal (coordinate-wise) feature selection. The method we present
is computationally feasible and statistically efficient with sample
complexity that primarily depends on the number of relevant
dimensions, and only logarithmically on the total number of features.
However, this goal was achieved by considering a very simple model -
a mixture of two non-spherical Gaussians with same covariance and
mixture weights. Extensions to allow uneven mixture weights, more than
two components and different covariance matrices are all crucial to make
the method practical, and will be the topic of a subsequent publication.
Theoretically, the bounds we have demonstrated can be tightened in a few places,
particularly for support recovery using a primal-dual witness argument.
Additionally, it will be interesting to demonstrate matching lower bounds
for this problem to establish optimality of the sample complexity.



\section*{Acknowledgements}
This research is supported in part by NSF awards IIS-1116458 and CAREER IIS-1252412.

\bibliographystyle{unsrt}
\bibliography{references}
\end{document}